\documentclass[12pt]{amsart}

\usepackage{amsmath,bm,amssymb,amsthm,textcomp}
\usepackage{enumitem}
\usepackage{mathtools}
\usepackage{hyperref}
\usepackage[numbers,sort&compress]{natbib}
\usepackage{tikz-cd}
\usetikzlibrary{cd}
\usepackage{comment}

\theoremstyle{plain}
\newtheorem{theorem}{Theorem}[section]
\newtheorem{lemma}[theorem]{Lemma}
\newtheorem{proposition}[theorem]{Proposition}
\newtheorem{corollary}[theorem]{Corollary}

\theoremstyle{definition}

\theoremstyle{remark}

\counterwithin{footnote}{subsection}

\counterwithin{equation}{section}

\allowdisplaybreaks

\begin{document}

\title[Continuity of the continued fraction mapping revisited]{Continuity of the continued fraction mapping revisited}

\author{Min Woong Ahn}
\address{Department of Mathematics, SUNY at Buffalo, Buffalo, NY 14260-2900, USA}
\email{minwoong@buffalo.edu}
\curraddr{Department of Mathematics Education, Silla University, 140, Baegyang-daero 700beon-gil, Sasang-gu, Busan 46958, Republic of Korea}
\email{minwoong@silla.ac.kr}

\date{\today}


\begin{abstract}
The continued fraction mapping maps a number in the interval $[0,1)$ to the sequence of its partial quotients. When restricted to the set of irrationals, which is a subspace of the Euclidean space $\mathbb{R}$, the continued fraction mapping is a homeomorphism onto the product space $\mathbb{N}^{\mathbb{N}}$, where $\mathbb{N}$ is a discrete space. In this short note, we examine the continuity of the continued fraction mapping, addressing both irrational and rational points of the unit interval.
\end{abstract}

\maketitle

\tableofcontents

\section{Introduction} \label{Introduction}

Given $x \in (0,1)$, the continued fraction expansion is expressed as:
\begin{align*}
x = [d_1(x), d_2(x), \dotsc] \coloneqq \cfrac{1}{d_1(x)+\cfrac{1}{d_2(x)+\cfrac{1}{d_3(x)+\ddots}}},
\end{align*}
where the $d_n(x)$ are the partial quotients. For irrationals, this expansion is infinite, forming sequences in $\mathbb{N}^{\mathbb{N}}$; for rational numbers, the expansion terminates after a finite number of terms, say $n$, with the condition $d_n(x) \geq 2$. This dual behavior motivates a detailed exploration of the continuity of the continued fraction mapping across the unit interval.

While the continued fraction mapping is a homeomorphism when restricted to the irrationals, its behavior at rational points is discontinuous in the classical sense. By redefining the codomain topology, one-sided continuity can be established at rational points. This paper develops such a topological framework, demonstrating how it simplifies analysis while maintaining rigor.

Topological insights into continued fractions have practical implications. For instance, Ridley and Petruska \cite{RP00} introduced the error-sum function of continued fractions and studied it directly on the unit interval, requiring heavy machinery and complex calculations. In \cite{Ahn23}, we investigated the error-sum function in the context of Pierce expansions, leveraging the topological viewpoint to simplify the analysis. This approach significantly reduced computational complexity while retaining mathematical rigor. Similarly, in the context of continued fractions, adopting a topological viewpoint holds significant potential for unifying the behaviors at rational and irrational points, offering deeper insights into the structure of the mapping and streamlining practical computations.

Throughout the paper, we denote by $\mathbb{N}$ the set of positive integers, $\mathbb{N}_0$ the set of non-negative integers, and $\mathbb{N}_\infty \coloneqq \mathbb{N} \cup \{ \infty \}$ the set of extended positive integers. Let $\mathbb{I} \coloneqq [0,1) \setminus \mathbb{Q}$, i.e., the set $\mathbb{I}$ is the set of irrationals in $[0,1)$. The Fibonacci sequence is denoted by $(F_n)_{n \in \mathbb{N}_0}$, with the recursive definition $F_{n+2} \coloneqq F_{n+1} + F_n$ for each $n \in \mathbb{N}_0$, where $F_0 \coloneqq 0$ and $F_1 \coloneqq 1$. Specifically, Binet's formula tells us that 
\begin{align} \label{Binet formula}
F_n = \frac{1}{\sqrt{5}} \left[ \left( \frac{1+\sqrt{5}}{2} \right)^n - \left( \frac{1-\sqrt{5}}{2} \right)^{n} \right] \quad \text{for each } n \in \mathbb{N}_0.
\end{align}

\section{Preliminaries of continued fractions} \label{Preliminaries}

In this section, we list some basic properties of continued fractions, most of which can be found in the classical textbook on number theory by Hardy and Wright \cite{HW08}, as well as in books more focused on the theory of continued fractions by Iosifescu and Kraaikamp \cite{IK02} and Khinchin \cite{Khi97}.

For each $x \in [0,1]$, define
\[
\overline{d}_1(x) \coloneqq 
\begin{cases} \lfloor 1/x \rfloor, &\text{if } x \neq 0; \\ \infty, &\text{if } x = 0, \end{cases}
\quad \text{ and } \quad
\overline{T}(x) \coloneqq
\begin{cases} 1/x - \lfloor 1/x \rfloor, &\text{if } x \neq 0; \\ 0, &\text{if } x = 0. \end{cases}
\]
Furthermore, for each $x \in [0,1]$ and $n \in \mathbb{N}$, let
\[
\overline{d}_n (x) \coloneqq \overline{d}_1 (\overline{T}^{n-1}(x)).
\]
Define $f \colon [0,1] \to \mathbb{N}_\infty^\mathbb{N}$ by
\[
f(x) \coloneqq (\overline{d}_k(x))_{k \in \mathbb{N}} = (\overline{d}_1(x), \overline{d}_2(x), \overline{d}_3(x), \dotsc)
\]
for each $x \in [0,1]$. Clearly, $f$ is well-defined by definition. Moreover, $f$ extends the continued fraction mapping, as it satisfies $f(x) = (d_k(x))_{k \in \mathbb{N}}$ for any irrational $x \in (0,1)$. For a rational $x \in (0,1)$ with the continued fraction expansion $x = [d_1(x), \dotsc, d_n(x)]$, $f(x)$ is given by $(d_1(x), \dotsc, d_n(x), \infty, \infty, \infty, \dotsc)$ with $d_n(x) \geq 2$. If we adopt the convention $1/\infty \coloneqq 0$, it is evident that $f$ consistently yields the partial quotients.

We shall consider a set closely related to the sequence of partial quotients of continued fractions. Put $\Sigma \coloneqq \Sigma_0 \cup \bigcup_{n \in \mathbb{N}} \Sigma_n \cup \Sigma_\infty$, where
\begin{align*}
\Sigma_0 &\coloneqq \{ \infty \}^\mathbb{N}, \\
\Sigma_n &\coloneqq \mathbb{N}^{\{ 1, \dotsc, n \}} \times \{ \infty \}^{\mathbb{N} \setminus \{ 1, \dotsc, n \}}, \quad \forall n \in \mathbb{N}, \\
\Sigma_\infty &\coloneqq \mathbb{N}^\mathbb{N}.
\end{align*}

For each $\sigma \coloneqq (\sigma_k)_{k \in \mathbb{N}} \in \mathbb{N}_\infty^{\mathbb{N}}$, put
\[
\widetilde{\varphi}_k (\sigma) \coloneqq [\underbrace{\sigma_1, \sigma_2, \dotsc, \sigma_k}_{\text{$k$ terms}}, \infty, \infty, \dotsc] = \cfrac{1}{\sigma_1 + \cfrac{1}{\sigma_2 + \cfrac{1}{\ddots+\cfrac{1}{\sigma_k}}}}
\]
for each $k \in \mathbb{N}_0$. We emphasize that $\widetilde{\varphi}_k$ is defined on the entire space $\mathbb{N}_\infty^{\mathbb{N}}$ rather than being restricted to $\Sigma_k$. For example, $\widetilde{\varphi}_{2}(3, \infty, 4, \infty, \infty, \dotsc) = [3, \infty, \infty, \dotsc] = 1/3$, and $\widetilde{\varphi}_{2}(\infty, \infty, \infty, \dotsc) = [\infty, \infty, \dotsc] = 0$. Now, define $\widetilde{\varphi} \colon \mathbb{N}_\infty^{\mathbb{N}} \to [0,1]$ by
\[
\widetilde{\varphi} (\sigma) \coloneqq \lim_{k \to \infty} \widetilde{\varphi}_k (\sigma)
\]
for each $\sigma \in \mathbb{N}_\infty^{\mathbb{N}}$. 

\begin{lemma} \label{observation lemma}
For any $\sigma \in \mathbb{N}_\infty^{\mathbb{N}} \setminus \Sigma_\infty$, there exists an $n \in \mathbb{N}$ such that $\widetilde{\varphi}(\sigma) = \widetilde{\varphi}_n (\sigma)$.
\end{lemma}

\begin{proof}
Let $\sigma \coloneqq (\sigma_k)_{k \in \mathbb{N}} \in \mathbb{N}_\infty^{\mathbb{N}} \setminus \Sigma_\infty$. By definition, we can find an $n \in \mathbb{N}$ such that $\sigma_n = \infty$. Then, $1/\sigma_n = 0 = 1/(\sigma_n+r)$ for any $r \in \mathbb{R}$ due to convention. This implies that
\begin{align*}
\widetilde{\varphi}_n (\sigma) 
&= [\sigma_1, \dotsc, \sigma_n, \infty, \infty, \dotsc] \\
&= [\sigma_1, \dotsc, \sigma_n, \sigma_{n+1}, \dotsc, \sigma_{n+j}, \infty, \infty, \dotsc] = \widetilde{\varphi}_{n+j} (\sigma)
\end{align*}
for any $j \in \mathbb{N}$. Hence the lemma.
\end{proof}

Upon combining Lemma \ref{observation lemma} and the following proposition, we deduce that the mapping $\widetilde{\varphi} \colon \mathbb{N}_\infty^{\mathbb{N}} \to [0,1]$ is well-defined. 

\begin{proposition} [See {\cite[Proposition~1.1.2]{IK02}}] \label{definition of phi n} 
For each $\sigma \in \Sigma_\infty$, the limit $\lim_{k \to \infty} \widetilde{\varphi}_k (\sigma)$ exists and belongs to $\mathbb{I}$.
\end{proposition}

We define $\varphi \colon \Sigma \to [0,1]$ by
\begin{align} \label{definition of phi}
\varphi(\sigma) \coloneqq \widetilde{\varphi} (\sigma) \quad \text{for each } \sigma \in \Sigma. 
\end{align}

Let $\sigma \in \Sigma \setminus \Sigma_0$ be arbitrary. Since $\varphi_n(\sigma) > 0$ for all $n \in \mathbb{N}$ by definition, we write
\[
\varphi_n (\sigma) \coloneqq p_n^* (\sigma)/q_n^* (\sigma),
\]
where $p_n^*(\sigma)$ and $q_n^*(\sigma)$ are coprime positive integers. Furthermore, define $p_{-1}^*(\sigma) \coloneqq 1$, $q_{-1}^*(\sigma) \coloneqq 0$, $p_0^*(\sigma) \coloneqq 0$, and $q_0^*(\sigma) \coloneqq 1$. (Notice that this resembles the definition of the {\em convergents} $p_n(x)/q_n(x)$, $n \in \mathbb{N}_0$, which are defined in terms of numbers in $[0,1)$.)

\begin{proposition} [See \cite{HW08, IK02, Khi97}] \label{metric properties}
Let $\sigma \coloneqq (\sigma_k)_{k \in \mathbb{N}} \in \Sigma \setminus \Sigma_0$. If $\sigma \in \Sigma_n$ for some $n \in \mathbb{N}$, then the following hold:
\begin{enumerate}[label = \upshape(\roman*), ref = \roman*, leftmargin=*, widest=iii]
\item \label{metric properties 1}
$q_k^* (\sigma) = \sigma_k q_{k-1}^* (\sigma) + q_{k-2}^* (\sigma)$ for each $k \in \{ 1, \dotsc, n \}$.
\item \label{metric properties 2}
$q_k^* (\sigma) \geq F_{k+1}$ for each $k \in \{ 1, \dotsc, n \}$.
\item \label{metric properties 3}
$|\varphi (\sigma) - \varphi_k (\sigma)| \leq 1/[q_k^* (\sigma) q_{k+1}^*(\sigma)]$ for each $k \in \mathbb{N}$.
\end{enumerate}
If $\sigma \in \Sigma_\infty$, then the equality and inequalities in parts (\ref{metric properties 1})--(\ref{metric properties 3}) hold for all $k \in \mathbb{N}$.
\end{proposition}

We remark that, in the traditional sense, every rational number in $(0,1)$ admits two distinct continued fraction expansions (see \cite[Theorem~162]{HW08} or \cite[p.~3]{IK02}). This is because, for any integer $k \geq 2$, we have $1/k = 1/[(k-1)+1/1]$ with $k-1 \in \mathbb{N}$. The following proposition rephrases this fact in terms of two maps $f \colon [0,1] \to \Sigma$ and $\varphi \colon \Sigma \to [0,1]$.

\begin{proposition} \label{inverse image of phi}
For each $x \in [0,1]$, the following hold:
\begin{enumerate}[label = \upshape(\roman*), ref = \roman*, leftmargin=*, widest=ii]
\item \label{inverse image of phi 1}
If $x \in \mathbb{I} \cup \{ 0, 1 \}$, then we have $\varphi^{-1}( \{ x \} ) = \{ \sigma \}$, where $\sigma \coloneqq f(x)$. Moreover, $\sigma \in \Sigma_\infty$ if $x \in \mathbb{I}$; $\sigma = (\infty, \infty, \dotsc) \in \Sigma_0$ if $x=0$; $\sigma = (1, \infty, \infty, \dotsc) \in \Sigma_1$ if $x = 1$.
\item \label{inverse image of phi 2}
If $x \in (0,1) \cap \mathbb{Q}$, then we have $\varphi^{-1}( \{ x \} ) = \{ \sigma, \sigma' \}$, where
\begin{align*}
\sigma &\coloneqq (\underbrace{\overline{d}_1(x), \overline{d}_2(x), \dotsc, \overline{d}_{n-1}(x)}_{\text{$n-1$ terms}}, \overline{d}_n(x), \infty, \infty, \dotsc) = f(x) \in \Sigma_n \cap f([0,1]), \\
\sigma' &\coloneqq (\underbrace{\overline{d}_1(x), \overline{d}_2(x), \dotsc, \overline{d}_{n-1}(x)}_{\text{$n-1$ terms}}, \overline{d}_n(x)-1, 1, \infty, \infty, \dotsc) \in \Sigma_{n+1} \setminus f([0,1]),
\end{align*}
for some $n \in \mathbb{N}$.
\end{enumerate}
\end{proposition}

For a given $\sigma \coloneqq (\sigma_k)_{k \in \mathbb{N}} \in \Sigma$, we define $\sigma^{(n)} \coloneqq (\tau_k)_{k \in \mathbb{N}}\in \Sigma$, for each $n \in \mathbb{N}_0$, by 
\[
\tau_k
\coloneqq 
\begin{cases}
\sigma_k, &\text{if } k \in \{ 1, \dotsc, n \}; \\
\infty, &\text{otherwise},
\end{cases}
\quad \text{i.e.,} \quad \sigma^{(n)} \coloneqq (\underbrace{\sigma_1, \dotsc, \sigma_n}_{\text{$n$ terms}}, \infty, \infty, \dots).
\]
Note that $\sigma^{(n)}$ is not in $\Sigma_n$ in general. For instance, if $\sigma \coloneqq (3, 5, \infty, \infty, \dotsc) \in \Sigma_2 \subseteq \Sigma$, we have $\sigma^{(3)} = \sigma \not \in \Sigma_3$.

Fix $n \in \mathbb{N}$ and $\sigma \in \Sigma_n$. Define the {\em cylinder set} associated with $\sigma$ by
\[
\Upsilon_\sigma \coloneqq \{ \upsilon \in \Sigma \colon \upsilon^{(n)} = \sigma \}.
\]
We define the {\em fundamental interval of length $n$} for all $n \geq 1$ and $\sigma \in \Sigma$ by
\[
I_n(\sigma) = I(\sigma_1, \dotsc, \sigma_n) \coloneqq \{ x \in [0,1] : \overline{d}_k(x) = \sigma_k, \, 1 \leq k \leq n \}.
\]

\section{Main results and proofs} \label{Main results and proofs}

In this section, we establish and prove the main results of this paper regarding the continuity of the continued fraction mapping.

Equip $\mathbb{N}$ with the discrete topology, and denote its one-point compactification by $\mathbb{N}_\infty$. Define $\rho \colon \mathbb{N}_\infty \times \mathbb{N}_\infty \to \mathbb{R}$ by
\[
\rho (m,n) \coloneqq \begin{cases} 1/m + 1/n, &\text{if } m \neq n; \\ 0, &\text{if } m = n, \end{cases}
\]
for each $m, n \in \mathbb{N}_\infty$.

\begin{proposition} [{\cite[Lemma 3.1]{Ahn23}}] \label{definition of rho}
The mapping $\rho \colon \mathbb{N}_\infty \times \mathbb{N}_\infty \to \mathbb{R}$ is a metric on $\mathbb{N}_\infty$ and induces the topology of the one-point compactification on $\mathbb{N}_\infty$.
\end{proposition}

Define $\rho^\mathbb{N} \colon \mathbb{N}_\infty^{\mathbb{N}} \times \mathbb{N}_\infty^{\mathbb{N}} \to \mathbb{R}$ by
\[
\rho^\mathbb{N} (\sigma, \tau) \coloneqq \sum_{k=1}^\infty \frac{1}{F_k^2} \rho (\sigma_{k}, \tau_{k})
\]
for each $\sigma \coloneqq (\sigma_k)_{k \in \mathbb{N}}$ and $\tau \coloneqq (\tau_k)_{k \in \mathbb{N}}$ in $\mathbb{N}_\infty^{\mathbb{N}}$. The following lemma is analogous to \cite[Lemma 3.3]{Ahn23}.

\begin{lemma} \label{Sigma is metric}
The mapping $\rho^\mathbb{N} \colon \mathbb{N}_\infty^{\mathbb{N}} \times \mathbb{N}_\infty^{\mathbb{N}} \to \mathbb{R}$ is a metric on $\mathbb{N}_\infty^{\mathbb{N}}$ and induces the product topology on $\mathbb{N}_\infty^{\mathbb{N}}$. Consequently, the metric space $(\mathbb{N}_\infty^{\mathbb{N}}, \rho^{\mathbb{N}})$ is compact.
\end{lemma}

\begin{proof}
For the first assertion, we note that $\sum_{k=1}^\infty 1/F_k^2 < \infty$, which follows directly from Binet's formula \eqref{Binet formula}. In fact, since $\phi \coloneqq (1+\sqrt{5})/2 >1$ and $(1-\sqrt{5})/2 = - 1/\phi$ has absolute value less than $1$, we have the convergence $(1/F_k^2)/(1/\phi^{2k}) \to 5$ as $k \to \infty$. Moreover, $\sum_{k=1}^\infty (1/\phi^{2k}) < \infty$ because it is a geometric series with ratio $1/\phi^2 <1$. By the limit comparison test, it follows that $\sum_{k=1}^\infty 1/F_k^2 < \infty$. We omit the remaining details, as the proof follows the same line as the standard proof that any countable product of metric spaces is metrizable. The second statement follows from Tychonoff's theorem, which tells us that the product space $\mathbb{N}_\infty^{\mathbb{N}}$ is compact. 
\end{proof}

Due to the preceding lemma, from now on, we may use $\mathbb{N}_\infty^{\mathbb{N}}$ to refer to both the product space and the metric space $(\mathbb{N}_\infty^{\mathbb{N}}, \rho^{\mathbb{N}})$.

Define $g \colon \mathbb{N}_\infty^{\mathbb{N}} \to \Sigma$ by
\[
g(\sigma) \coloneqq
\begin{cases}
\sigma^{(n)}, &\text{if } \sigma^{(n)} \in \Sigma_n \text{ and } \sigma^{(n+1)} \not \in \Sigma_{n+1} \text{ for some } n \in \mathbb{N}_0; \\
\sigma, &\text{if } \sigma \in \Sigma_\infty,
\end{cases}
\]
for each $\sigma \coloneqq (\sigma_k)_{k \in \mathbb{N}} \in \mathbb{N}_\infty^{\mathbb{N}}$. It is not hard to verify that $g$ is well-defined and that $g$ is surjective. Under the mapping $g$, each sequence forgets all terms following the first $\infty$, if any, and replaces the forgotten terms with $\infty$'s.

\begin{lemma} \label{varphi tilde as composition}
We have $\widetilde{\varphi} = \varphi \circ g$ on $\mathbb{N}_\infty^{\mathbb{N}}$.
\end{lemma}

\begin{proof}
Note first that, for any $\sigma \in \mathbb{N}_\infty^{\mathbb{N}}$, we have $g(\sigma) \in \Sigma$, and so $\varphi (g(\sigma)) = \widetilde{\varphi} (g(\sigma))$ by definition of $\varphi$ in \eqref{definition of phi}. If $\sigma \in \Sigma_\infty$, then $g(\sigma) = \sigma$ by definition, and hence, $\widetilde{\varphi} (g(\sigma)) = \widetilde{\varphi} (\sigma)$. Now, assume $\sigma \in \mathbb{N}_\infty^{\mathbb{N}} \setminus \Sigma_\infty$, and put $\sigma \coloneqq (\sigma_k)_{k \in \mathbb{N}}$. Then, we can find the unique $n \in \mathbb{N}_0$ such that $\sigma^{(n)} \in \Sigma_n$ and $\sigma_{n+1} = \infty$, so that $g(\sigma) = \sigma^{(n)} \in \Sigma_n$. It follows by definition that $\widetilde{\varphi}(\sigma) = \widetilde{\varphi}_n(\sigma) = \widetilde{\varphi}_n(g(\sigma)) = \widetilde{\varphi}(g(\sigma))$, and this completes the proof.
\end{proof}

The following lemma shows that the mapping $\widetilde{\varphi} \colon \mathbb{N}_\infty^{\mathbb{N}} \to [0,1]$ is Lipschitz, and therefore continuous.

\begin{lemma} \label{phi is Lipschitz}
For any $\sigma, \tau \in \mathbb{N}_\infty^{\mathbb{N}}$, we have $| \widetilde{\varphi} (\sigma) - \widetilde{\varphi} (\tau)| \leq \rho^{\mathbb{N}} (\sigma, \tau)$.
\end{lemma}

\begin{proof}
Let $\sigma \coloneqq (\sigma_k)_{k \in \mathbb{N}}, \tau \coloneqq (\tau_k)_{k \in \mathbb{N}} \in \mathbb{N}_\infty^{\mathbb{N}}$. If $\sigma = \tau$, the inequality holds trivially; hence, we suppose that $\sigma$ and $\tau$ are distinct. If $\sigma_1 \neq \tau_1$, then
\[
|\widetilde{\varphi} (\sigma) - \widetilde{\varphi} (\tau)| \leq |\widetilde{\varphi} (\sigma)| + |\widetilde{\varphi} (\tau)| \leq \frac{1}{\sigma_1} + \frac{1}{\tau_1} = \frac{1}{F_1^2} \rho (\sigma_1, \tau_1) \leq \rho^{\mathbb{N}} (\sigma, \tau).
\]
Assume that $\sigma$ and $\tau$ share the initial block of length $n \in \mathbb{N}$, i.e., $\sigma^{(n)} = \tau^{(n)}$ and $\sigma_{n+1} \neq \tau_{n+1}$. If $\sigma^{(n)} = \tau^{(n)} \not \in \Sigma_n$, then $\sigma_j = \tau_j = \infty$ for some $j \in \{ 1, \dotsc, n \}$, and it follows that $\widetilde{\varphi} (\sigma) = \widetilde{\varphi} (\tau)$; thus, the inequality holds. Now, assume $\sigma^{(n)} = \tau^{(n)} \in \Sigma_n$. Since $\sigma_{n+1} \neq \tau_{n+1}$, at least one of $\sigma_{n+1}$ and $\tau_{n+1}$ is finite. Without loss of generality, we may further assume $\sigma_{n+1} \neq \infty$, so that 
\begin{align} \label{g sigma in Sigma j}
g(\sigma) \in \Sigma_l \quad \text{for some } l \in \mathbb{N}_\infty \setminus \{ 1, \dotsc, n, n+1 \}. 
\end{align}
Write $q_k^* = q_k^* (\sigma^{(n)}) = q_k^* (\tau^{(n)})$ for each $k \in \{ 0, 1, \dotsc, n \}$. We consider two cases separately according as $\tau_{n+1} = \infty$ or $\tau_{n+1} \neq \infty$.

{\sc Case I}.
Assume $\tau_{n+1} = \infty$. Then, by definitions, we have
\begin{gather} \label{equalities from definition}
\begin{gathered}
\widetilde{\varphi}(\sigma^{(n)}) = [\sigma_1, \dotsc, \sigma_n, \infty, \infty, \dotsc] = \varphi_n (g(\sigma)), \\
\widetilde{\varphi}(\tau)
= [\tau_1, \dotsc, \tau_n, \infty, \tau_{n+2}, \dotsc] 
= [\tau_1, \dotsc, \tau_n, \infty, \infty, \dotsc] = \widetilde{\varphi}(\tau^{(n)}).
\end{gathered}
\end{gather}
Hence, we find that
\begin{align*}
&| \widetilde{\varphi} (\sigma) - \widetilde{\varphi} (\tau)| \\
&\hspace{0.4em} = | ( \widetilde{\varphi} (\sigma) - \widetilde{\varphi} (\sigma^{(n)}) ) -  (\widetilde{\varphi} (\tau) - \widetilde{\varphi} (\tau^{(n)}) ) | && \text{since $\sigma^{(n)} = \tau^{(n)}$} \\
&\hspace{0.4em} = |\varphi (g (\sigma)) - \varphi_n (g(\sigma)) | && \text{by Lemma \ref{varphi tilde as composition} and \eqref{equalities from definition}} \\
&\hspace{0.4em} \leq \frac{1}{q_n^*(g(\sigma))  q_{n+1}^* (g(\sigma))} && \text{by \eqref{g sigma in Sigma j} and Proposition \ref{metric properties}(\ref{metric properties 3})} \\
&\hspace{0.4em} = \frac{1}{q_n^*(g(\sigma))} \cdot \frac{1}{\sigma_{n+1} q_n^*(g(\sigma)) + q_{n-1}^*(g(\sigma))} && \text{by Proposition \ref{metric properties}(\ref{metric properties 1})} \\
&\hspace{0.4em} = \frac{1}{q_n^*} \cdot \frac{1}{\sigma_{n+1} q_n^* + q_{n-1}^*} && \text{since $\sigma^{(n)} = (g(\sigma))^{(n)} \in \Sigma_{n}$} \\
&\hspace{0.4em} < \frac{1}{(q_n^*)^2} \frac{1}{\sigma_{n+1}} 
=  \frac{1}{(q_n^*)^2} \left( \frac{1}{\sigma_{n+1}} + \frac{1}{\tau_{n+1}} \right) &&\text{since $1/\tau_{n+1}=0$ by assumption} \\
&\hspace{0.4em} \leq \frac{1}{F_{n+1}^2} \left( \frac{1}{\sigma_{n+1}} + \frac{1}{\tau_{n+1}} \right)  &&\text{by Proposition \ref{metric properties}(\ref{metric properties 2})} \\
&\hspace{0.4em} = \frac{1}{F_{n+1}^2} \rho (\sigma_{n+1}, \tau_{n+1})
\leq \rho^{\mathbb{N}} (\sigma, \tau),
\end{align*}
as desired.

{\sc Case II}.
Assume $\tau_{n+1} \neq \infty$. Then, $g(\tau) \in \Sigma_{l'}$ for some $l' \in \mathbb{N}_\infty \setminus \{ 1, \dotsc, n, n+1 \}$. An argument similar to the one in the preceding case results in
\begin{align*}
| \widetilde{\varphi} (\sigma) - \widetilde{\varphi} (\tau)| 
&= | ( \widetilde{\varphi} (\sigma) - \widetilde{\varphi} (\sigma^{(n)}) ) -  (\widetilde{\varphi} (\tau) - \widetilde{\varphi} (\tau^{(n)}) ) | \\
&\leq | \varphi (g (\sigma)) - \varphi_n (g (\sigma))| + | \varphi (g (\tau)) - \varphi_n (g (\tau))| \\
&\leq \frac{1}{q_n^*(g(\sigma))  q_{n+1}^* (g(\sigma))} + \frac{1}{q_n^*(g(\tau))  q_{n+1}^* (g(\tau))} \\
&= \frac{1}{q_n^*} \left( \frac{1}{\sigma_{n+1} q_n^* + q_{n-1}^*} + \frac{1}{\tau_{n+1} q_n^* + q_{n-1}^*} \right) \\
&< \frac{1}{(q_n^*)^2} \left( \frac{1}{\sigma_{n+1}} + \frac{1}{\tau_{n+1}} \right) \\
&\leq \frac{1}{F_{n+1}^2} \left( \frac{1}{\sigma_{n+1}} + \frac{1}{\tau_{n+1}} \right) 
= \frac{1}{F_{n+1}^2} \rho (\sigma_{n+1}, \tau_{n+1})
\leq \rho^{\mathbb{N}} (\sigma, \tau).
\end{align*}
This completes the proof of the lemma.
\end{proof}

Let $\sim_g$ be the binary relation on $\mathbb{N}_\infty^{\mathbb{N}}$ given by $\sigma \sim_g \tau$ if and only if $g(\sigma) = g(\tau)$ for $\sigma, \tau \in \mathbb{N}_\infty^{\mathbb{N}}$. Evidently, $\sim_g$ is an equivalence relation on $\mathbb{N}_\infty^{\mathbb{N}}$. Note that the map $\widetilde{g} \colon \mathbb{N}_\infty^{\mathbb{N}} / \sim_g \to \Sigma$ given by $\widetilde{g} ([\sigma]) = g(\sigma)$, for each $[\sigma] \in \mathbb{N}_\infty^{\mathbb{N}} / \sim_g$, is a bijection. We equip $\mathbb{N}_\infty^{\mathbb{N}} / \sim_g$ with the quotient topology, and define 
\[
\mathcal{T}_\Sigma \coloneqq \{ \widetilde{g}(U) : \text{$U$ is open in $\mathbb{N}_\infty^{\mathbb{N}} / \sim_g$} \}.
\]

\begin{lemma} \label{g tilde is homeomorphism}
The set $\mathcal{T}_\Sigma$ defines a topology on $\Sigma$, and $\widetilde{g} \colon \mathbb{N}_\infty^{\mathbb{N}} / \sim_g \to (\Sigma, \mathcal{T}_\Sigma)$ is a homeomorphism.
\end{lemma}

\begin{proof}
The lemma is clear from the definitions.
\end{proof}

\begin{lemma} \label{commutative diagram}
Let $\pi_g \colon \mathbb{N}_\infty^{\mathbb{N}} \to \mathbb{N}_\infty^{\mathbb{N}} / \sim_g$ denote the canonical projection. Then, the following diagram commutes:
\[
\begin{tikzcd}
\mathbb{N}_\infty^{\mathbb{N}} \arrow{r}{\widetilde{\varphi}} \arrow[swap]{d}{\pi_g} \arrow{dr}{g} & {[0,1]} \\
\mathbb{N}_\infty^{\mathbb{N}} / \sim_g \arrow[swap]{r}{\widetilde{g}} & (\Sigma, \mathcal{T}_\Sigma) \arrow[swap]{u}{\varphi}
\end{tikzcd}
\]
In particular, $g \colon \mathbb{N}_\infty^{\mathbb{N}} \to (\Sigma, \mathcal{T}_\Sigma)$ is a continuous surjection. 
\end{lemma}

\begin{proof}
The result follows from Lemma \ref{varphi tilde as composition} and the paragraph preceding Lemma \ref{g tilde is homeomorphism}.
\end{proof}

\begin{lemma} \label{closed equivalence relation}
The graph of the equivalence relation $\sim_g$, i.e., the set $R_g \coloneqq \{ (\sigma, \tau) \in \mathbb{N}_\infty^{\mathbb{N}} \times \mathbb{N}_\infty^{\mathbb{N}} : \sigma \sim_g \tau \}$, is closed in the product space $\mathbb{N}_\infty^{\mathbb{N}} \times \mathbb{N}_\infty^{\mathbb{N}}$.
\end{lemma}

\begin{proof}
Since the product space $\mathbb{N}_\infty^{\mathbb{N}} \times \mathbb{N}_\infty^{\mathbb{N}}$ is metrizable as a finite product of the metric space $\mathbb{N}_\infty^{\mathbb{N}}$ (Lemma \ref{Sigma is metric}), it suffices to show that any convergent sequence in $R_g$ has its limit in $R_g$. 

Let $((\bm{\sigma}_k, \bm{\tau}_k))_{k \in \mathbb{N}}$ be a convergent sequence in $R_g$ with the limit $(\sigma, \tau) \in \mathbb{N}_\infty^{\mathbb{N}} \times \mathbb{N}_\infty^{\mathbb{N}}$. Suppose $(\sigma, \tau) \in \mathbb{N}_\infty^{\mathbb{N}} \times \mathbb{N}_\infty^{\mathbb{N}} \setminus R_g$ for the sake of contradiction. Then, $\sigma \not \sim_g \tau$, or, equivalently, $g(\sigma) \neq g(\tau)$. Since $g(\sigma)$ and $g(\tau)$ are elements of $\Sigma$, we may write $g(\sigma) = ((g(\sigma))_j)_{j \in \mathbb{N}}$ and $g(\tau) = ((g(\tau))_j)_{j \in \mathbb{N}}$. Put $M \coloneqq \min \{ j \in \mathbb{N} : (g(\sigma))_j \neq (g(\tau))_j \} < \infty$. Without loss of generality, we may assume that $(g(\sigma))_M \neq \infty$. Then, $(g(\sigma))_j \in \mathbb{N}$ for all $j \in \{ 1, \dotsc, M-1 \}$ and $g(\sigma) \in \Sigma_i$ for some $i \in \mathbb{N}_\infty \setminus \{ 1, \dotsc, M \}$. There are two cases to consider according as $\tau_M \neq \infty$ or $\tau_M = \infty$.

{\sc Case I}. Assume $\tau_M \neq \infty$. Since the first $M$ terms of $g(\sigma)$ are all finite, we have $\sigma^{(M)} = (g(\sigma))^{(M)} \in \Sigma_M$ by definition of $g$; similarly, $\tau^{(M)} = (g(\tau))^{(M)} \in \Sigma_M$. Note that $\sigma_M \neq \tau_M$ by definition of $M$. Since $\bm{\sigma}_k \to \sigma$ and $\bm{\tau}_k \to \tau$ as $k \to \infty$ in the product space $\mathbb{N}_\infty^{\mathbb{N}}$, we know that $(\bm{\sigma}_k)^{(M)} = \sigma^{(M)} \in \Sigma_M$ and $(\bm{\tau}_k)^{(M)} = \tau^{(M)} \in \Sigma_M$ for all large enough $k$. Then, for such $k$, it must be that $(g(\bm{\sigma}_k))_M = \sigma_M$ and $(g(\bm{\tau}_k))_M = \tau_M$. But $\sigma_M \neq \tau_M$, and it follows that $g(\bm{\sigma}_k) \neq g(\bm{\tau}_k)$, i.e., $\bm{\sigma}_k \not \sim_g \bm{\tau}_k$, for all large enough $k$, a contradiction. 

{\sc Case II}. Assume $\tau_M = \infty$. Then, $(g(\bm{\sigma}_k))_M = \sigma_M$ for all large enough $k$ as in the previous case, while $(g(\bm{\tau}_k))_M \to \infty$ as $k \to \infty$. In particular, $(g(\bm{\tau}_k))_M > \sigma_M$ for all sufficiently large $k$, and therefore, $g(\bm{\sigma}_k) \neq g(\bm{\tau_k})$ for all such $k$. Thus, $\bm{\sigma}_k \not \sim_g \bm{\tau}_k$ for all but finitely many $k$'s, contradicting that $(\bm{\sigma}_k, \bm{\tau}_k) \in R_g$ for all $k \in \mathbb{N}$.

This proves that the limit $(\sigma, \tau)$ must be in $R_g$ in either case. Hence the lemma.
\end{proof}

We state some standard facts in general topology in the following two propositions.

\begin{proposition}[See {\cite[p.~105, Proposition 8]{Bou67a}}] \label{first Bourbaki proposition}
Let $X$ be a compact space and $\sim$ an equivalence relation on $X$. If $X/\sim$ denotes the quotient space, then the following statements are equivalent:
\begin{enumerate} [label = \upshape(\roman*), ref = \roman*, leftmargin=*, widest=iii]
\item
The graph of $\sim$, i.e., the set $\{ (x, y) \in X \times X : x \sim y \}$, is closed in the product space $X \times X$.
\item
The canonical projection $\pi \colon X \to X/\sim$ is a closed mapping.
\item
$X/\sim$ is Hausdorff.
\end{enumerate}
\end{proposition}

\begin{proposition}[See {\cite[p.~159, Proposition 17]{Bou67b}}] \label{second Bourbaki proposition}
Let $X$ be a compact metrizable space and $\sim$ an equivalence relation on $X$. If the quotient space $X / \sim$ is Hausdorff, then $X / \sim$ is compact metrizable.
\end{proposition}

\begin{lemma} \label{Sigma is compact metrizable}
The topological space $(\Sigma, \mathcal{T}_\Sigma)$ is compact metrizable.
\end{lemma}

\begin{proof}
Since the graph of $\sim_g$ is closed in the product space $\mathbb{N}_\infty^{\mathbb{N}} \times \mathbb{N}_\infty^{\mathbb{N}}$ (Lemma \ref{closed equivalence relation}), we infer, in view of Proposition \ref{first Bourbaki proposition}, that the quotient space $\mathbb{N}_\infty^{\mathbb{N}} / \sim_g$ is Hausdorff. Now, since $\mathbb{N}_\infty^{\mathbb{N}}$ is compact metrizable (Lemma \ref{Sigma is metric}) and since the quotient space $\mathbb{N}_\infty^{\mathbb{N}}/\sim_g$ is Hausdorff, Proposition \ref{second Bourbaki proposition} tells us that $\mathbb{N}_\infty^{\mathbb{N}} / \sim_g$ is compact metrizable. Therefore, we conclude that $(\Sigma, \mathcal{T}_\Sigma)$ is compact metrizable, as a homeomorphic space to $\mathbb{N}_\infty^{\mathbb{N}}/\sim_g$ via the mapping $\widetilde{g} \colon \mathbb{N}_\infty^{\mathbb{N}} / \sim_g \to (\Sigma, \mathcal{T}_\Sigma)$ (Lemma \ref{g tilde is homeomorphism}).
\end{proof}

\begin{lemma} \label{cylinder set is compact}
For each $n \in \mathbb{N}$ and $\sigma \in \Sigma_n$, the subspaces $\Upsilon_\sigma$ and $\Sigma \setminus \Upsilon_\sigma$ of $(\Sigma, \mathcal{T}_\Sigma)$ are compact metrizable.
\end{lemma}

\begin{proof}
We first note that, by Lemma \ref{closed equivalence relation} and Proposition \ref{first Bourbaki proposition}, the canonical projection $\pi_g \colon \mathbb{N}_\infty^{\mathbb{N}} \to \mathbb{N}_\infty^{\mathbb{N}} / \sim_g$ is a closed mapping. But $g = \widetilde{g} \circ \pi_g$ by Lemma \ref{commutative diagram}, where $\widetilde{g}$ is a homeomorphism (Lemma \ref{g tilde is homeomorphism}), and this implies that $g$ is also a closed mapping.

Now, fix $n \in \mathbb{N}$ and $\sigma \coloneqq (\sigma_k)_{k \in \mathbb{N}} \in \Sigma_n$. Consider the set
\[
g^{-1}(\Upsilon_\sigma) = \{ \sigma_1 \} \times \dotsm \times \{ \sigma_n \} \times \mathbb{N}_\infty^{\mathbb{N} \setminus \{ 1, \dotsc, n \}},
\]
which is clopen in $\mathbb{N}_\infty^{\mathbb{N}}$. Then, $g^{-1}(\Upsilon_\sigma)$ and $\mathbb{N}_\infty^{\mathbb{N}} \setminus g^{-1}(\Upsilon_\sigma)$ are closed in $\mathbb{N}_\infty^{\mathbb{N}}$, and hence, $\Upsilon_\sigma = g(g^{-1}(\Upsilon_\sigma))$ and $\Sigma \setminus \Upsilon_\sigma = g(\mathbb{N}_\infty^{\mathbb{N}} \setminus g^{-1}(\Upsilon_\sigma))$ are closed in $(\Sigma, \mathcal{T}_\Sigma)$ as $g$ is a closed mapping by the preceding paragraph. Since $(\Sigma, \mathcal{T}_\Sigma)$ is compact metrizable (Lemma \ref{Sigma is compact metrizable}), its closed subspaces are compact metrizable. This completes the proof of the lemma.
\end{proof}

\begin{theorem} \label{phi is continuous}
The mapping $\varphi \colon (\Sigma, \mathcal{T}_\Sigma) \to [0,1]$ is continuous.
\end{theorem}

\begin{proof}
It is clear that $\widetilde{\varphi} \colon \mathbb{N}_\infty^{\mathbb{N}} \to [0,1]$ is continuous if and only if $\varphi \colon (\Sigma, \mathcal{T}_\Sigma) \to [0,1]$ is continuous. Indeed, for any open subset $U$ of $[0,1]$, we have $\widetilde{\varphi}^{-1}(U) = \pi^{-1} \circ \widetilde{g}^{-1} \circ \varphi^{-1}(U)$ by Lemma \ref{commutative diagram}, so that
\begin{align*}
\widetilde{\varphi}^{-1}(U) \text{ is open in } \mathbb{N}_\infty^{\mathbb{N}}
&\iff \widetilde{g}^{-1} \circ \varphi^{-1}(U) \text{ is open in } \mathbb{N}_\infty^{\mathbb{N}}/\sim_g \\
&\iff \varphi^{-1}(U) \text{ is open in } (\Sigma, \mathcal{T}_\Sigma),
\end{align*}
where the first equivalence follows from the definition of the quotient topology and the second from the fact that $\widetilde{g}$ is a homeomorphism (Lemma \ref{g tilde is homeomorphism}). But then, the continuity of $\widetilde{\varphi}$ (Lemma \ref{phi is Lipschitz}) implies the continuity of $\varphi$.
\end{proof}

We are now ready to present the continuity theorem for the extended continued fraction mapping, which is an analogue of \cite[Lemmas 3.9 and 3.10]{Ahn23}.

\begin{theorem} \label{f continuity theorem}
For the continuity of the mapping $f \colon [0,1] \to (\Sigma, \mathcal{T}_\Sigma)$, the following hold:
\begin{enumerate}[label = \upshape(\roman*), ref = \roman*, leftmargin=*, widest=ii]
\item \label{f is continuous at irrational}
$f$ is continuous at every irrational point and at two rational points $0$ and $1$.
\item \label{f is discontinuous at rational}
$f$ is one-sided continuous only at every rational point in $(0,1)$; precisely, if we let $x \in (0,1) \cap \mathbb{Q}$ and put $\varphi^{-1} ( \{ x \} ) = \{ \sigma, \tau \}$, then
\[
\lim_{\substack{t \to x \\ t \in I_n(\sigma)}} f(t) = \sigma
\quad \text{and} \quad
\lim_{\substack{t \to x \\ t \not \in I_n(\sigma)}} f(t) = \tau,
\]
where $\sigma \in \Sigma_n$ for some $n \in \mathbb{N}$.
\end{enumerate}
\end{theorem}

\begin{proof}
(\ref{f is continuous at irrational})
In view of Proposition \ref{inverse image of phi}(\ref{inverse image of phi 1}), it is sufficient to show that $f$ is continuous at points $x \in [0,1]$ where $\varphi^{-1}(\{ x \})$ is a singleton. Suppose, for contradiction, that $f$ is not continuous at such a point $x$. Then, there exists a sequence $(x_k)_{k \in \mathbb{N}}$ converging to $x$ such that the corresponding images $f(x_k)$ fail to converge to $f(x)$. Using the compact metrizability of $\Sigma$ (Lemma \ref{Sigma is compact metrizable}), we can extract a subsequence of $(f(x_k))_{k \in \mathbb{N}}$ that converges to some $\tau \in \Sigma$. The continuity of $\varphi$ (Theorem \ref{phi is continuous}) implies that $x = \varphi (\tau)$, and the singleton assumption forces $\tau = f(x)$. This contradicts the initial choice of the sequence $(x_k)_{k \in \mathbb{N}}$, establishing the continuity of $f$ at such points.

(\ref{f is discontinuous at rational})
Suppose, to the contrary, that $\lim_{\substack{t \to x \\ t \in I_n(\sigma)}} f(t) \neq \sigma$. Then, there exists a sequence $(x_k)_{k \in \mathbb{N}}$ in $I_n(\sigma)$ converging to $x$ such that $(f(x_k))_{k \in \mathbb{N}}$ does not converge to $\sigma$. By the compact metrizability of $\Upsilon_\sigma$ (Lemma \ref{cylinder set is compact}), we extract a subsequence of $(f(x_k))_{k \in \mathbb{N}}$ converging to some $\upsilon \in \Upsilon_\sigma$. The continuity of $\varphi$ (Theorem \ref{phi is continuous}) implies that $x = \varphi (\upsilon)$. The doubleton assumption forces $\upsilon = \sigma$ or $\upsilon = \tau$. Since $\tau \not \in \Upsilon_\sigma$, it must be that $\upsilon = \sigma$, which contradicts the initial assumption about the sequence $(x_k)_{k \in \mathbb{N}}$. Thus, $\lim_{\substack{t \to x \\ t \in I_n(\sigma)}} f(t) = \sigma$.

The proof for the second equality is similar. One needs to consider sequences in $\Sigma \setminus \Upsilon_\sigma$ and uses the compact metrizability of this subspace (Lemma \ref{cylinder set is compact}).
\end{proof}

\begin{corollary}
The subspace $\mathbb{I}$ of $[0,1]$ and the subspace $\Sigma_\infty$ of $(\Sigma, \mathcal{T}_\Sigma)$ are homeomorphic via the restriction of $f \colon [0,1] \to (\Sigma, \mathcal{T}_\Sigma)$ to $\mathbb{I}$ and the restriction of $\varphi \colon (\Sigma, \mathcal{T}_\Sigma) \to [0,1]$ to $\Sigma_\infty$, serving as each other's continuous inverses.
\end{corollary}

\begin{proof}
It is a standard fact that the sets $\mathbb{I}$ and $\Sigma_\infty$ are in a bijective relation via the continued fraction mapping. Therefore, to obtain the corollary, we only need to combine Theorems \ref{f continuity theorem}(\ref{f is continuous at irrational}) and \ref{phi is continuous}.
\end{proof}

In conclusion, the results presented in this paper provide a topological framework that unifies the continuity behavior of the continued fraction mapping at both irrational and rational points. As mentioned in the introduction, this framework can be leveraged to reestablish the results of \cite{RP00}, specifically regarding the continuity of the error-sum function at irrational points and the one-sided continuity of the function at rational points, without relying on the complicated computations traditionally required. While the detailed derivation of these results using the topological approach is beyond the scope of this paper, readers are encouraged to explore this connection further to appreciate the utility of the framework introduced here.

\section*{Acknowledgements}

I would like to express sincere gratitude to the referees for their thorough and insightful review of this paper. Their detailed comments and suggestions have immensely improved the manuscript, particularly in terms of clarity, rigor, and readability.

\end{document}